\documentclass{amsart}
\usepackage{latexsym,amssymb,amsmath,amsthm,amscd}
\numberwithin{equation}{section}

\newtheorem{theorem}{Theorem}[section]
\newtheorem{lemma}[theorem]{Lemma}

\newtheorem{proposition}[theorem]{Proposition}

\theoremstyle{definition}

\theoremstyle{remark}

\begin{document}

\title{An observation of the subspaces of ${\mathcal S}'$}

\author{Yoshihiro Sawano}

\keywords{Lizorkin distributions, topologial dual, homogeneous Besov spaces}

\maketitle

\address{%
Department of Mathematics and Information Science, 
Tokyo Metropolitan University, 
1-1 Minami-Ohsawa, Hachioji, Tokyo 192-0397, Japan.}%

\begin{abstract}
The spaces
${\mathcal S}'/{\mathcal P}$
equipped with the quotient topology
and
${\mathcal S}'_\infty$
equipped with the weak-* topology
are known to be homeomorphic,
where ${\mathcal P}$ denotes the set of
all polynomials.
The proof is a combination 
of the fact in the textbook
by Treves and the well-known bipolar theorem.
In this paper by extending slightly the idea
employed in \cite{NNS15},
we give an alternative proof of this fact
and then we extend this proposition
so that we can include some related function spaces.
\end{abstract}

{\bf 2010 Classification:46A04, 46A20, 46A22, 42B35}

\section{Introduction}

It is useful to consider the quotient spaces
of ${\mathcal S}'$ or ${\mathcal D}'$
when we consider the homogeneous function spaces.
Usually such a quotient space can be identified
with some dual spaces as the following theorem shows:
\begin{theorem}\label{thm:main}
Let $X$ be a locally convex $($Hausdorff$)$ space
whose topology is given by a family of semi-norms
$\{p_\lambda\}_{\lambda \in \Lambda}$.
Equip 
$X^*$ with the weak-$*$ topology.
Let $V$ be a closed subspace
of $X^*$.
Define the orthogonal space $X_V$ to $V$ by:
\[
X_V \equiv \bigcap_{x^* \in V}\ker(x^*)
\]
and 
equip $X_V$ with the topology induced
by $X$.
Then the topological dual
$X_V^*$ is isomorphic to $X^*/V$
equiped with the quotient topology.
\end{theorem}

The proof of Theorem \ref{thm:main} is 
a combination of \cite[Propositions 35.5 and 35.6]{Treves67}
and the bipolar theorem.
\begin{theorem}{\rm\cite[Propositions 35.5 and 35.6]{Treves67}}
Let $X$ be a locally convex Hausdorff space,
and let $N$ be a closed linear subspace of $X$.
Then the kernel of the restriction $X'$ to $N'$
is 
\[
N^\circ=\bigcap_{n \in N}\{x^* \in X'\,:\,<x',n>=0\}.
\]
Furthermore, its quotient mapping
is a homeomorphism from
$X'/N^\circ$ to $N'$.
\end{theorem}

\begin{theorem}[Bipolar theorem, {\rm \cite[p. 126, Theorem]{Schaefer}}]
Let $X$ be a Hausdorff topological vector space.
Let $V$ be a closed subspace
of $X^*$ equipped with the weak-* topology .
Define
\[
{}^\circ V
\equiv
\bigcap_{v^* \in V}
\{x \in X\,:\,\langle v^*,x \rangle=0\}
=
\bigcap_{v^* \in V}
\ker(v^*).
\]
Then
\[
V=\bigcap_{x \in {}^\circ V}
\{v^* \in X^*\,:\,\langle v^*,x \rangle=0\}
(\equiv ({}^\circ V)^\circ).
\]
\end{theorem}
If we let
$V={\mathcal P}(\subset {\mathcal S}')$,
the linear space of all polynomials,
then we can show that
$V$ is a closed subspace of
${\mathcal S}'$.
One of the ways to check this is to use the Fourier transform.
In fact,
$f \in {\mathcal S}'$ belongs to ${\mathcal P}$
if and only if
the Fourier transform is supported in $\{0\}$.

The isomorphism
$X^*/V \to X_V^*$
is given as follows:
Let $R$ be the natural restriction mapping
$R:X^* \in f \mapsto f|X_V \in X_V^*$.
Denote by $\iota:X_V \to X$ the natural inclusion.
Then $\iota$ and $R$ are dual to each other
and $R$ is clearly continuous.

The aim of this paper is to give out
an alternative proof of Theorem \ref{thm:main}.
We organize this short note as follows:
We prove Theorem \ref{thm:main} in Section \ref{s2}.
After collecting some preliminary facts
in Section \ref{s2.1} we plan to prove Theorem \ref{thm:main}.
We shall show that
$\ker(R)=V$
in Section \ref{s2.2} which is essentially
the bipolar theorem,
that $R$ is surjective
in Section \ref{s2.3}
and that
$R$ is an open mapping
in Section \ref{s2.4}.

We compare Theorem \ref{thm:main}
with the existing results in Section \ref{s3}.
We devote Sections \ref{s3.1}, \ref{s3.2}, \ref{s3.3}
and \ref{s3.4}
to the application of Theorem \ref{thm:main}
to the spaces 
${\mathcal S}_\infty'$,
${\mathcal S}_m'$,
${\mathcal S}_{\rm e}'$
and
${\mathcal D}'$
respectively.
The definition of
${\mathcal S}_m'$ and
${\mathcal S}_{\rm e}'$ will be given in Sections \ref{s3.2}
and \ref{s3.3}, respectively.
For a topological space $Y$ and its dual $Y^*$,
we write
$<y^*,y> \equiv y^*(y)$
for the coupling of $y \in Y$ and $y^* \in Y^*$.

\section{Proof of Theorem \ref{thm:main}}
\label{s2}

\subsection{A reduction and preliminaries}
\label{s2.1}

Let $X$ be a locally convex space
whose topology is given by a family of semi-norms
$\{p_\lambda\}_{\lambda \in \Lambda}$.
Let us set
\[
{\mathcal P}
\equiv 
\left\{\sum_{\lambda \in \Lambda_0}
a_\lambda p_\lambda\,:\,
\Lambda_0\mbox{ is a finite subset of }\Lambda
\mbox{ and }
\{a_\lambda\}_{\lambda \in \Lambda_0}
\subset {\mathbb N}
\right\}.
\]
Let $O \subset X^*$
be an open set.
Then there exists $q \in {\mathcal P}$
such that
\[
\{x \in X\,:\,q(x)<1\}
\subset
O.
\]
Therefore by replacing
$\{p_\lambda\}_{\lambda \in \Lambda}$
with
${\mathcal P}$,
we can assume that 
for any open set $O$
there exists $\lambda(O) \in \Lambda$ such that
\begin{equation}\label{eq:160119-2}
\{x \in X\,:\,p_{\lambda(O)}(x)<1\}
\subset
O.
\end{equation}

We invoke the propositions
concerning the Hahn-Banach extension.
First, we recall the following fact:
\begin{proposition}[Geometric form, {\rm \cite[p. 46]{Schaefer}}]\label{prop:1}
Let $M$ be a linear subspace 
in a topological vector space $L$
and let $A$ be a non-empty convex, open subset
of $L$, not intersecting $M$.
Then there exists a closed hyperplane in $L$,
containing $M$ and not intersecting $A$.
\end{proposition}

Next,
we recall the Mazur theorem.
\begin{proposition}[Analytic form, {\rm \cite[p. 108, Theorem 3]{Yoshida}}]\label{thm:H-B-1}
Let $X$ be a locally convex linear topological space,
and $M$ be a closed convex subset
of $X$ such that
$a \cdot m \in M$
whenever $|a| \le 1$ and $m \in M$.
Then for any $x_0 \in X \setminus M$
there exists a continuous linear functional
$f$ on $X$ such that 
$f(x_0)>1 \ge |f(x)|$
for all $x_0 \in M$.
\end{proposition}

When $M$ is a linear subspace in the above,
$f(x)=0$
for all $x \in M$.
Thus, we can deduce the following well-known version:
\begin{proposition}[Analytic form]\label{thm:H-B}
Let $X$ be a topological vector space
and let $Y$ be a closed linear space.
Then for any continuous linear functional $\ell_Y$
and $x \in X \setminus Y$
there exists a continuous linear functional
$\ell_X$ such that $\ell_X|Y=\ell_Y$ and that
$\ell_X(x)=0$.
\end{proposition}

\subsection{The kernel of $R$}
\label{s2.2}

We now specify $\ker(R)$.
It is easy to see that
$V \subset \ker(R)$ and
that $V$ and $\ker(R)$ are weak-* closed.
Assume that $V$ and $\ker(R)$ are different.
Then by Proposition \ref{prop:1},
we can find a continuous linear functional 
$\Phi^{**}:X^* \to {\mathbb C}$
such that
$V \subset \ker(\Phi^{**})$
and
\begin{equation}\label{eq:160119-1}
\ker(R) \cap \ker(\Phi^{**})^{\rm c} \ne \emptyset.
\end{equation}
Since $\Phi^{**}:X^* \to {\mathbb C}$ is continuous,
we have
\begin{eqnarray*}
&&\{x^* \in X^*\,:\,
|<x^*,x_1>|<1, 
|<x^*,x_2>|<1, 
\ldots, 
|<x^*,x_k>|<1
\}\\
&&\quad\subset
\{x^* \in X^*\,:\,|<\Phi^{**},x^*>|<1\}
\end{eqnarray*}
for some $x_1,x_2,\ldots,x_k \in X$.
This means that
\[
\bigcap_{j=1}^k \ker(Q_{x_j}) 
\subset
\ker(\Phi^{**}),
\]
where $Q:X \ni x \mapsto Q_x \in X^{**}$
is a natural inclusion.
By the Helly theorem,
we see that
\[
\Phi^{**}=\sum_{j=1}^k a_j Q_{x_j} \in X^{**}.
\]
Let $x^* \in V$ be arbitrary.
Then we have
\[
\left<x^*,\sum_{j=1}^k a_j x_j\right>
=
<\Phi^{**},x^*>=0,
\]
since $x^* \in \ker(\Phi^{**})$.
This means that
\[
\sum_{j=1}^k a_j x_j \in X_V.
\]
Let $x^* \in \ker(R)$.
Then
\[
<\Phi^{**},x^*>
=
<Q_{\sum_{j=1}^k a_j x_j},x^*>
=
\left<x^*,\sum_{j=1}^k a_j x_j\right>
=0,
\]
since $x^*|X_V=0$ and $\sum_{j=1}^k a_j x_j \in X_V$.
Thus,
$\ker(R) \subset \ker(\Phi^{**})$.
This contradicts (\ref{eq:160119-1}).

\subsection{The surjectivity of $R:X \to X_V$}
\label{s2.3}

Let $z^* \in X_V^*$.
Then
$|<z^*,z>| \le p_\lambda(z)$
for some $\lambda \in \Lambda$
by (\ref{eq:160119-2}).
Use the Hahn-Banach theorem
of analytic form to have a continuous linear
functional $x^*$ on $X$ which extends
$z^*$.
Then $R(x^*)=z^*$.

\subsection{The openness of $R$}
\label{s2.4}

We need the following lemma:
\begin{lemma}\label{lem:1}
Let
$z_1,z_2,\ldots,z_k \in X_V$
and
$x_1,x_2,\ldots,x_l \in X$.
Assume that the system
$\{[x_1],[x_2],\ldots,[x_l]\}$
is linearly independent 
in $X/X_V$.
Then for all
$z^* \in X_V$
satisfying
$|<z^*,z_j>|<1$
for all $j=1,2,\ldots,k$,
we can find
$x^* \in X_V$
so that
$z^*$ is a restriction of $x^*$
and that
$<x^*,x_j>=0$
for all $j=1,2,\ldots,l$.
\end{lemma}

\begin{proof}
We know that any linear space $X$ has a norm
$\|\cdot\|_X$
although it is not necessarily compatible
with its original topological structure of $X$.
For example, choose a Hamel basis
$\{x_\theta\}_{\theta \in \Theta}$
and define
\[
\left\|\sum_{\theta \in \Theta_0}a_\theta x_\theta
\right\|_{X}
\equiv 
\sum_{\theta \in \Theta_0}|a_\theta|
\]
for any finite set $\Theta_0$.

Observe
$
\sum_{j=1}^l a_j x_j
\notin X_V
$
for any
$(a_1,a_2,\ldots,a_l) \ne (0,0,\ldots,0)$,
which yields
$x^*_{(a_1,a_2,\ldots,a_l)} \in X^*$
such that
\begin{equation}\label{eq:160131-1}
\left<
x^*_{(a_1,a_2,\ldots,a_l)},
\sum_{j=1}^l a_j x_j
\right>=1.
\end{equation}
Since
$x^*_{(a_1,a_2,\ldots,a_l)} \in X^*$
is a continuous linear functional,
we can find an index
$\lambda(a_1,a_2,\ldots,a_l) \in \Lambda$
such that
\begin{equation}\label{eq:160131-2}
\{x \in X\,:\,p_{\lambda(a_1,a_2,\ldots,a_l)}(x)<1\}
\subset
\{x \in X\,:\,|<x^*_{(a_1,a_2,\ldots,a_l)},x>|<1\}.
\end{equation}
Write
\begin{equation}\label{eq:160131-3}
U_{(a_1,a_2,\ldots,a_l)}
\equiv 
\left\{
(b_1,b_2,\ldots,b_l) \in {\mathbb C}^n\,:\,
p_{\lambda(a_1,a_2,\ldots,a_l)}
\left(\sum_{j=1}^l (a_j-b_j) x_j\right)<1\right\}.
\end{equation}
Since
$S^{2l-1}\equiv \{
(b_1,b_2,\ldots,b_l) \in {\mathbb C}^l\,:\,
|b_1|^2+|b_2|^2+\cdots+|b_l|^2=1\}$
is a compact set,
we can find a finite covering
$\{U_{(a_1,a_2,\ldots,a_l)}\}_{(a_1,a_2,\ldots,a_l) \in A}$
of $S^{2l-1}$,
where $A$ is a finite subset
of $S^{2l-1}$.

Let us denote by
$X_{A}$ and $X_{A,V}$
the completion 
of 
$X$ and $X_V$ with respect to the norm
\[
\|\cdot\|_{A}
\equiv 
\|\cdot\|_X+
\sum_{(a_1,a_2,\ldots,a_l) \in A}
p_{\lambda(a_1,a_2,\ldots,a_l)}(\cdot),
\]
respectively.
Then we can extend
$x^*_{(a_1,a_2,\ldots,a_l)}$
to a continuous linear functional
${\mathfrak X}^*_{(a_1,a_2,\ldots,a_l)}$
to $X_{A,V}$.

Note that
for any
$(b_1,b_2,\ldots,b_l) \in {\mathbb C}^n \setminus
(0,0,\ldots,0)$,
there exists
$(a_1,a_2,\ldots,a_l) \in A$
such that
\[
\frac{1}{\sqrt{|b_1|^2+|b_2|^2+\cdots+|b_l|^2}}
(b_1,b_2,\ldots,b_l)
\in U(a_1,a_2,\ldots,a_l).
\]
For such an element
$(a_1,a_2,\ldots,a_l) \in A$,
we deduce
\[
x^*_{\lambda(a_1,a_2,\ldots,a_l)}
\left(\sum_{j=1}^k a_j x_j\right) \ne 0
\]
from (\ref{eq:160131-1})--(\ref{eq:160131-3}).
Since
${\mathfrak X}^*_{\lambda(a_1,a_2,\ldots,a_l)}(x)=0$
for any $x \in X_{V,A}$,
we see that
$\{[x_1],[x_2],\ldots,[x_l]\}$
is a linearly independent system
in $X_A/X_{A,V}$.
Then use the Hahn-Banach theorem
of geometric form
(see Proposition \ref{thm:H-B}). Then
we obtain a continuous functional
${\mathfrak X}^*:X_A\to {\mathbb C}$
so that
${\mathfrak X}^*(x_j)=0$
for $j=1,2,\ldots,l$.
If we set ${\mathfrak X}^*|X=x^*$,
then we have the desired result.
\end{proof}

The following proposition shows that
$R$ is an open mapping:
\begin{proposition}
Let
$z_1,z_2,\ldots,z_k \in X_V$
and
$x_1,x_2,\ldots,x_L \in X$.
Then there exists $r>1$
depending only on 
$z_1,z_2,\ldots,z_k \in X_V$
and
$x_1,x_2,\ldots,x_L \in X$
such that
for all
$z^* \in X_V$
satisfying
$|<z^*,z_j>|<1$
for all $j=1,2,\ldots,k$,
we can find
$x^* \in X_V$
so that
$z^*$ is a restriction of $x^*$
and that
$|<x^*,x_j>|<r$
for all $j=1,2,\ldots,L$.
Namely, the range 
\[
\bigcap_{j=1}^k
\{x^* \in X^*\,:\,|<x^*,z_j>|<1\}
\cap
\bigcap_{j=1}^l
\{x^* \in X^*\,:\,|<x^*,x_j>|<r\}
\]
by $R$ contains
\[
\bigcap_{j=1}^k
\{x^* \in X^*\,:\,|<x^*,z_j>|<1\}.
\]
In particular, the range 
\[
\bigcap_{j=1}^k
\{x^* \in X^*\,:\,|<x^*,z_j>|<1\}
\cap
\bigcap_{j=1}^l
\{x^* \in X^*\,:\,|<x^*,x_j>|<1\}
\]
by $R$ contains
\[
\bigcap_{j=1}^k
\{x^* \in X^*\,:\,|<x^*,z_j>|<r^{-1}\}.
\]
\end{proposition}

\begin{proof}
Let us assume that
$\{[x_1],[x_2],\ldots,[x_l]\}$
is a maximal linearly independent family 
in $X/X_V$,
where $l \le L$.
Then for any 
$j \in(l,L] \cap {\mathbb N}$
and
$k \in[1,l] \cap {\mathbb N}$,
we can find
$\mu_{j k} \in {\mathbb C}$
and
$\tilde{z}_j \in X_V$
such that
\[
x_j=\tilde{z}_j+\sum_{m=1}^l \mu_{j m}x_m.
\]
Let 
\[
r\equiv 1+\max_{j=1,2,\ldots,l}|<z^*,\tilde{z}_j>|
\]
Note that
\[
|<z^*,z_j>|,
|<z^*,\tilde{z}_m>|<r
\]
for all 
$j=1,2,\ldots,k$
and
$m=l+1,l+2,\ldots,L$.
According to Lemma \ref{lem:1},
we can find
$x^* \in X^*$
so that $x^*$ is an extension of $z^*$
and that
\[
<z^*,x_j>=0, \quad j \in (l,L].
\]
Thus, we obtain the desired result.
\end{proof}

\section{Applications}
\label{s3}

\subsection{Schwartz space}
\label{s3.1}

The Schwartz space ${\mathcal S}$ is defined
to be the set of all $\Phi \in C^\infty$
for which the semi-norm $p_N(\Phi)$ is finite for all 
$N \in {\mathbb N}_0\equiv \{0,1,\ldots\}$,
where
\[
p_N(\Phi)
\equiv 
\sum_{|\alpha| \le N}
\sup_{x \in {\mathbb R}^n}
(1+|x|)^N|\partial^\alpha \Phi(x)|.
\]
The space ${\mathcal S}_\infty$ is the set of all
$\Phi \in {\mathcal S}$ for which 
\[
\int_{{\mathbb R}^n}x^\alpha\Phi(x)\,dx=0
\]
for all 
$\alpha \in {\mathbb N}_0^{\ n}$.

The topological dual of 
${\mathcal S}$
and 
${\mathcal S}_\infty$
are denoted by 
${\mathcal S}'$
and 
${\mathcal S}_\infty'$,
respectively.
The elements
in
${\mathcal S}'$
and 
${\mathcal S}_\infty'$
are called
Schwarz distributions
and 
Lizorkin distributions, 
respectively.
Equip 
${\mathcal S}'$
and 
${\mathcal S}_\infty'$
with the weak-* topology.

Since ${\mathcal S}_\infty$ is continuously embedded
into ${\mathcal S}$,
the dual operator $R$, called the restriction,
is continuous from ${\mathcal S}'$
to ${\mathcal S}'_\infty$.
We can generalize the following fact and refine the proof:
\begin{theorem}{\rm \cite[Theorem 6.18]{NNS15}}\label{thm:150301-1}
The restriction mapping 
$R:F \in {\mathcal S}' \mapsto F|{\mathcal S}_\infty'
\in {\mathcal S}'$ is open,
namely the image $R(U)$ is open in ${\mathcal S}_\infty'$
for any open set $U$ in ${\mathcal S}'$.
\end{theorem}
The statement can be found in \cite{Triebel1},
where Triebel applied this theorem
to the definition of homogeneous function spaces.
Note also that Holschneider
considered Theorem \ref{thm:150301-1}
in the context of wavelet analysis
in \cite[Theorem 24.0.4]{Holschneider95},
where he applied 
a general result
\cite[Propositions 35.5 and 35.6]{Treves67}
to this special setting.
We can find the proof of Theorem \ref{thm:150301-1}
in \cite[Proposition 8.1]{YSY10-2}.
But there is a gap in Step 4,
where the openness of $R$ is proved using the closed graph theorem.
It seems that the closed graph theorem is not applicable
to the space ${\mathcal S}'$.
Our proof reinforces Step 4 in the proof of
\cite[Proposition 8.1]{YSY10-2}.

According to the proof of Theorem \ref{thm:main},
there is no need to use the Fourier transform.

\subsection{The space ${\mathcal S}_m'$}
\label{s3.2}

We recall the definition 
of ${\mathcal S}'/{\mathcal P}_m$,
where ${\mathcal P}_m$ denotes the set of
all polynomials of degree less than or equal to $m$.
Following Bourdaud \cite{Bourdaud11},
we denote by ${\mathcal S}_m$
the orthogonal space of ${\mathcal P}_m$
in ${\mathcal S}$
and by ${\mathcal S}'_m$
its topological dual.
See
\cite{LSUYY12,YaYu10-2,YYZ12,ZYY14-1}
for applications to homogeneous function spaces
defined in \cite{YaYu08-2,YaYu10-2,YSY10-2}.

\subsection{The Hasumi space ${\mathcal S}_{\rm e}'$}
\label{s3.3}

Let
$N \in {\mathbb N}$ and $\alpha \in {\mathbb N}_0{}^n$.
Write temporarily
$\varphi_{(N;\alpha)}(x)\equiv e^{N|x|} \partial^\alpha \varphi(x)$
$(x \in {\mathbb R}^n)$
for $\varphi \in C^\infty$.
Define
${\mathcal S}_{\rm e}$
as follows:
\[
{\mathcal S}_{\rm e}\equiv
\bigcap_{N \in {\mathbb N}, \alpha \in {\mathbb N}_0{}^n}
\left\{\varphi \in C^\infty \, : \,
\varphi_{(N;\alpha)} \in L^\infty
\right\}.
\]
The topological dual is denoted by
${\mathcal S}_{\rm e}'$
is called the Hasumi space \cite{Hasumi61}.
An analogy to the spaces
${\mathcal S}'$
and
${\mathcal S}_m'$
is available.
We refer
to \cite{Rychkov01}
for function spaces contained 
in ${\mathcal S}_{\rm e}'$.

\subsection{The space ${\mathcal D}'$}
\label{s3.4}

A similar thing to Sections \ref{s3.1}--\ref{s3.3} applies 
to ${\mathcal D}'$.
If we define
\[
{\mathcal D}_m
\equiv
\left\{
\varphi \in {\mathcal D}\,:\,
\int_{{\mathbb R}^n}x^\alpha \varphi(x)\,dx=0
\mbox{ for all }\alpha \in {\mathbb N}_0{}^n
\mbox{ with }|\alpha| \le m
\right\},
\]
then we have
\[
{\mathcal D}_m' \sim {\mathcal D}'/{\mathcal P}_m.
\]

\section{Acknowledgement}

The author is thankful to Professor Kunio Yoshino
and the anonymous reviewer for their pointing
out the references
\cite{Holschneider95,Treves67}.

\end{document}